\documentclass{article}
\usepackage[dvipdfmx]{graphicx}
\usepackage{amsmath}
\usepackage{amsthm}
\usepackage{amsfonts}
\usepackage{amssymb}
\usepackage{latexsym}
\usepackage{authblk}
\def\qed{\hfill $\Box$}
\makeatletter
\renewenvironment{proof}[1][\proofname]{\par
  \normalfont
  \topsep6\p@\@plus6\p@ \trivlist
  \item[\hskip\labelsep{\bfseries #1}\@addpunct{\bfseries.}]\ignorespaces
}{
  \endtrivlist
}
\renewcommand{\proofname}{Proof}
\makeatother
\begin{document}
\title{The $l_q$ consistency of the Dantzig Selector for Cox's Proportional Hazards Model}
\author{Kou Fujimori}
\author{Yoichi Nishiyama}
\affil{Waseda University}
\date{}
\maketitle
\theoremstyle{plain}
\newtheorem{defi}{Definition}[section]
\newtheorem{lem}[defi]{Lemma}
\newtheorem{thm}[defi]{Theorem}
\newtheorem{rem}[defi]{Remark}
\newtheorem{fact}[defi]{Fact}
\newtheorem{ex}[defi]{Example}
\newtheorem{prop}[defi]{Proposition}
\newtheorem{cor}[defi]{Corollary}
\newtheorem{condition}[defi]{Condition}
\newtheorem{assumption}[defi]{Assumption}
\newcommand{\sign}{\mathop{\rm sign}}
\newcommand{\conv}{\mathop{\rm conv}}
\newcommand{\argmax}{\mathop{\rm arg~max}\limits}
\newcommand{\argmin}{\mathop{\rm arg~min}\limits}
\newcommand{\argsup}{\mathop{\rm arg~sup}\limits}
\newcommand{\arginf}{\mathop{\rm arg~inf}\limits}
\begin{abstract}
The Dantzig selector for the 
proportional hazards model proposed by D.R. Cox is studied in a high-dimensional and sparse setting. 
We prove the $l_q$ consistency for all $q \geq 1$ of some estimators based on the compatibility factor, 
the weak cone invertibility factor, 
and the restricted eigenvalue for certain deterministic matrix which approximates 
the Hessian matrix of log partial likelihood. 
Our matrix conditions for these three factors are weaker than those of previous researches.
\end{abstract}
\section{Introduction}
The proportional hazards model proposed by Cox (1972) is widely used in survival analysis. 
Using the counting process approach, Andersen and Gill (1982) proved that the maximum partial likelihood estimator 
satisfies the consistency and the asymptotic normality when the number of covariates $p$ is fixed. 
In this paper, we deal with a high-dimensional and sparse setting, $i.e.$, the case where $p$ is much larger than $n$ and 
the number $S$ of non-zero components of the true parameter is relatively small. 
In this setting, the maximum partial likelihood estimator may not behave well, so we need to consider
other estimation procedures. 
For example, the penalized methods such as LASSO (Tibshirani (1997), Huang {\it et al}.\ (2013), Bradic {\it et al}.\ (2013), among others) and SCAD (Fan and Li (2002) and Bradic {\it et al}.\ (2011)) are studied by many researchers. 
Instead of the penalized methods, in this paper we apply a relatively new method called the Dantzig selector for the proportional hazards model. 

The Dantzig selector (DS) was proposed by Cand\'{e}s and Tao (2007) in a high-dimensional and sparse setting for the
linear regression model:
\[Y = Z \beta + \epsilon.\]
The Dantzig selector returns the estimator $\hat{\beta}$ for $\beta$ defined as follows:
\[\hat{\beta} :=\argmin \limits_{\beta \in \mathcal{C}} \|\beta\|_1,\quad \mathcal{C} :=\{\beta \in \mathbb{R}^p : |Z^j{^T}(Y-Z\beta)| \leq \lambda\},\]
where $\|\cdot\|_1$ is $l_1$-norm, $Z^j$ denotes the $j$-th column of the design matrix $Z$ and $\lambda \geq0$ 
is a suitable constant.   
Assume that $\epsilon$ is a Gaussian error term for simplicity. 
When $\lambda = 0$, the DS returns to the MLE. For $\lambda >0$, the DS searches for the sparsest $\beta$ 
within the given distance of the MLE. 
Notice that this method has a good potential to be applied for other models. 
Antoniadis {\it et al}.\ (2010) proposed the Survival Dantzig selector (SDS), which is an application of the DS 
for the proportional hazards model, and studied the $l_2$ consistency of the SDS under the UUP condition 
which is the matrix condition used by Cand\'es and Tao (2007). The purpose of this paper is to prove the 
$l_q$ consistency for $q \geq 1$ under some weaker matrix conditions. 
Our way of proofs are similar to those of Huang {\it et al}.\ (2013) for LASSO in the proportional hazards model. 
They introduced the compatibility factor, the weak invertibility factor, and the restricted eigenvalue for 
Hessian matrix of log partial likelihood. Since the Hessian is a random matrix, these factors are 
random variables. They also derived some conditions to treat them as deterministic constants. 
In contrast, we will define these factors for a deterministic matrix, which approximates the 
Hessian matrix to clarify the proofs. 

This paper is organized as follows. The settings for
the proportional hazards model and an estimation procedure are given in Section $2$. 
In Section $3$, we introduce some matrix conditions to derive the consistency.
Main results are presented in Section $4$. 

Throughout this paper, for every $q \in [1,\infty]$, we denote by $\|\cdot\|_q$ the $l_q$-norm of $\mathbb{R}^p$-vector, which is defined as follows:
\begin{align*}
\|v\|_q &= \left(\sum \limits_{j=1}^p |v_j|^q \right)^{\frac{1}{q}},\quad q < \infty ; \\
\|v\|_{\infty} &= \sup \limits_{1 \leq j \leq p} |v_j|.
\end{align*}
Moreover, for a given $m \times n$ matrix $A$, where $m,\ n \in \mathbb{N}$, we define $\|A\|_{\infty}$ by
\[\|A\|_{\infty} := \sup \limits_{1 \leq i \leq m} \sup \limits_{1 \leq j \leq n} |A_i^j|,\]
where $A_i^j$ denotes the $(i,j)$-component of the matrix $A$.
\section{Model set up and estimation procedure}
\subsection{Model set up}
Let $T$ be a survival time, and $C$ a censoring time, which are positive real-valued random variables on a probability space 
$(\Omega,\mathcal{F},P)$.  
The survival time $T$ is assumed to be conditionally independent of a censoring time $C$ given the  multidimensional covariate $Z$.
For every $n \in \mathbb{N}$, we observe the data $\{(X_i,D_i,Z_i)\}_{i=1}^n$, which are $i.i.d.$ copies of $(X,D,Z)$, where $D = 1_{\{T \leq C\}}$. 
Assume that the dimension of covariates $p = p_n$ depends on $n$. In particular, we assume that $p_n \gg n$.
Let $N_i (t) = 1_{\{X_i \leq t,\ D_i =1\}}$ be counting processes based on these data, and 
$Y_i (t) = 1_{\{X_i \geq t\}}$ at risk processes for every $i= 1,2,\ldots, n$. 
Assume that the sample paths $t \leadsto N_i(t),\ i=1,2,\ldots,n$, have no simultaneous jumps.  
We introduce the filtration $\{\mathcal{F}_t\}_{t \geq 0}$ defined by
\begin{align*}
\mathcal{F}_t  &:= \sigma \{ N_i(u),\ Y_i(u),\ 0 \leq u \leq t,\ Z_i,\ i = 1,2,\ldots,n \}.
\end{align*}
Assume that $N_i$ has the intensity $\lambda_i$ of the form 
\[\lambda_i (t,Z_i) = Y_i (t) \alpha_0 (t) \exp(Z_i ^T \beta_0),\]
where $\alpha_0$ is a baseline hazard function which is regarded as a nuisance parameter, and 
$\beta_0 \in \mathbb{R}^{p_n}$ is the unknown true parameter. We are interested in the estimation problem of 
$\beta_0$ in a high dimensional and sparse setting, $i.e.$, in the case where $\beta_0 \in \mathbb{R}^{p_n}$ has $S$ non-zero components.
Using the Doob-Meyer decomposition theorem, we have that the processes
\[M_i (t) := N_i (t) - \int_0^t \lambda_i (u,Z_i)du,\ i = 1,2,\ldots,n,\]
are independent $\{\mathcal{F}_t\}_{t \geq 0}$ square-integrable martingales. Their predictable quadratic variations are given by
\[\langle M_i,M_i \rangle(t) = \int_0^t \lambda_i (u,Z_i) du.\]
We fix the time interval $[0,\tau]$. Following Cox (1972), we introduce the log partial likelihood defined by
\[C_n(\beta) := \sum \limits_{i=1}^{n} \int_{0}^{\tau} \{Z_i ^T \beta -\log (S_n^0 (\beta,u))\} dN_i (u),\]
where $S_n^0 (\beta,u) := \sum \limits_{i=1}^n Y_i (u)\exp (Z_i ^T \beta)$.\ 
Furthermore, we define the $p_n$ dimensional vector $U_n(\beta)$, and the $p_n \times p_n$ matrix $J_n(\beta)$ as
\begin{align*}
U_n(\beta) &:=\dot{l}_n(\beta) = \frac{1}{n} \sum \limits_{i =1}^{n} \int_{0}^{\tau} \left(Z_i - \frac{S_n^1}{S_n^0}(\beta,u)\right)dN_i(u),\\
J_n(\beta) &:=-\ddot{l}_n(\beta) =\frac{1}{n} \int_{0}^{\tau} \left[\frac{S_n^2}{S_n^0} (\beta,u) - \left(\frac{S_n^1}{S_n^0}\right)^{\otimes2} (\beta,u)\right]d\bar{N}(u),
\end{align*}
where $\bar{N} (u):=\sum \limits_{i=1}^n N_i (u),\ l_n(\beta):=C_n(\beta)/n,\ S_n^1 (\beta,u):=\sum \limits_{i=1}^n Y_i (u)\exp (Z_i ^T \beta)Z_i$,
 and $S_n ^2 (\beta,u):=\sum \limits_{i=1}^n Y_i (u)\exp (Z_i ^T \beta)Z_i ^{\otimes 2}$. Note that 
$J_n(\beta_0)$ is a nonnegative definite matrix.
In particular, for the true parameter $\beta = \beta_0$ and for all $j = 1,2,\ldots,p_n$, 
we can see that $U_n^j(\beta_0),\ j = 1,2,\ldots,p_n$ are the terminal values of 
martingales $\{U_n^j (\beta_0,t)\}_{t \in [0,\tau]}$ given by:
\[U_n^j (\beta_0,t) = \frac{1}{n} \sum \limits_{i=1}^{n} \int_{0}^{t}\left( Z_i ^j - \frac{S_{n}^{1j}}{S_n^0} (\beta_0,u) \right)dM_i(u).\] 
Hereafter, we assume the following conditions.
\begin{assumption}
\begin{description}
\item[$(i)$]
There exists a positive constant $K_1$ such that
\[\sup \limits_{i,j} |Z^j_i| < K_1.\]
\item[$(ii)$]
The baseline hazard function $\alpha_0$ is integrable on $[0,\tau],\ i.e.$, 
\[\int_0^\tau \alpha_0(t) dt < \infty.\]
\item[$(iii)$]
There exist $\mathbb{R}$-valued function $s_n^0(\beta,t)$, $\mathbb{R}^{p_n}$-valued function $s_n^1(\beta,t)$, and $p_n \times p_n$ matrix-valued function $s_n^2(\beta,t)$ such that for $l = 0,1,2$, 
\[\sup \limits_{\beta \in \mathbb{R}^{p_n}} \sup \limits_{t \in [0,\tau]} \left\|\frac{1}{n} S_n^l (\beta,t) - s_n^l (\beta,t) \right\|_{\infty} \rightarrow^p 0.\]
\item[$(iv)$]
For $l = 0,1,2$, the functions $(\beta,t) \mapsto s_n^l(\beta,t)$ are uniformly continuous. Moreover, they satisfy the following conditions:
\[\limsup_{n \rightarrow \infty}\sup \limits_{\beta \in \mathbb{R}^{p_n}} \sup \limits_{t \in [0,\tau]} \|s_n^l(\beta,t)\|_{\infty} < \infty,\]
\[\liminf_{n \rightarrow \infty}\inf \limits_{\beta \in \mathbb{R}^{p_n}} \inf \limits_{t \in [0,\tau]} s_n^0 (\beta,t) >0.\]
\item[$(v)$]
Assume that $S$ is a constant not depending on $n$. Define the $p_n \times p_n$ matrix 
$I_n(\beta)$ of rank $S$ by
\[I_n(\beta) := \int_{0}^{\tau} \left[\frac{s_n^2}{s_n^0}(\beta,u) - \left(\frac{s_n^1}{s_n^0}\right)^{\otimes2}(\beta,u)\right]
s_n^0(\beta_0,u)\alpha_0 (u) du.\]
When we compute $s_n^0,\ s_n^1,\ s_n^2$ only with respect to the components of the true $S$ dimensional vector $\beta_0$, 
we have the $S \times S$ sub-matrix $\mathcal{I}_n(\beta_0)$ of $I_n(\beta_0)$. 
We assume that $\mathcal{I}_n(\beta_0)$ is positive definite, and $I_n(\beta)$ for all $\beta \in \mathbb{R}^{p_n}$ are nonnegative definite matrices.
\end{description}
\end{assumption}
\subsection{Estimation procedure}
We define the estimator $\hat{\beta}_n$ of $\beta_0$ as
\begin{equation}
\hat{\beta}_n := \argmin \limits_{\beta \in \mathcal{B}_n} \|\beta\|_1 ,
\label{eq1}
\end{equation} 
where $\mathcal{B}_n := \{\beta \in \mathbb{R}^{p_n}:\|U_n(\beta)\|_\infty \leq \gamma \}$, 
and $\gamma \geq 0$ is a suitable constant.
We call the estimator $\hat{\beta}_n$ the Dantzig Selector for Proportional Hazards model (DSfPH). 
Note that this estimator is called the Survival Dantzig Selector (SDS) by Antoniadis {\it et al}.\ (2010).

\section{Matrix conditions}
In this section, we will discuss some matrix conditions to derive the theoretical results for DSfPH $\hat{\beta}_n$. 
Hereafter, we write $T_0$ for the support of $\beta_0$, $i.e.$, 
\[T_0 := \{j : \beta_{0j} \not= 0\} .\]
To begin with, we introduce the following three factors $(A),\ (B)$ and $(C)$, all of which are used by 
Huang {\it et al}.\ (2013) 
for LASSO in Cox's proportional hazards model.
\begin{defi}
For every index set $T \subset \{1,\ 2,\ \cdots,\ p_n\}$ and $h \in \mathbb{R}^{p_n}$, 
$h_T$ is a $\mathbb{R}^{|T|}$ dimensional sub-vector of $h$ constructed by extracting the components of $h$ corresponding to the indices in $T$. Define the set $C_T$ by 
\[C_T := \{h \in \mathbb{R}^{p_n} : \|h_{T^c}\|_1 \leq \|h_{T}\|_1\} .\]
We introduce the following three factors.
\begin{description}
\item[$(A)$ Compatibility factor]
\[\kappa(T_0;I_n(\beta_0)) := \inf \limits_{0 \not = h \in C_{T_0}} \frac{S^{\frac{1}{2}} (h^T I_n(\beta_0) h)^{\frac{1}{2}}}{\|h_{T_0}\|_1}.\]
\item[$(B)$ Weak cone invertibility factor]
\[F_q(T_0;I_n(\beta_0)) :=  \inf \limits_{0 \not = h \in C_{T_0}} \frac{S^{\frac{1}{q}} (h^T I_n(\beta_0) h)^{\frac{1}{2}}}{\|h_{T_0}\|_1 \|h\|_q},\quad q \geq 1.\]
\item[$(C)$ Restricted eigenvalue]
\[RE(T_0;I_n(\beta_0)) :=  \inf \limits_{0 \not = h \in C_{T_0}} \frac{ (h^T I_n(\beta_0) h)^{\frac{1}{2}}}{\|h\|_2}.\]
\end{description}
\end{defi}
As mentioned in the Introduction, Huang {\it et al}. (2013) defined these factors for the random matrix $J_n(\beta_0)$, and derived some conditions to treat them as deterministic constants. On the other hand, we define them not for   $J_n(\beta_0)$, but for the deterministic matrix $I_n(\beta_0)$, since we will prove that 
$\|I_n(\beta_0)-J_n(\beta_0)\|_{\infty} = o_p(1)$ in Section $4$ of this paper. 

There exist other matrix conditions for DSfPH such as the UUP condition, which is used in Cand\'es and Tao (2007) 
and Antoniadis {\it et al}.\ (2010). 
To discuss the relationship between the UUP condition and our conditions, let us introduce some objects. 

Note that there exists a matrix $A$ such that  $A^T A = I_n(\beta)$, because $I_n(\beta)$ is a nonnegative definite matrix. 
Given a index set $T \subset \{1,2,\ldots,p_n\}$, 
we write $A_T$ for the $p_n \times |T|$ matrix 
constructed by extracting the columns of $A$ corresponding to the indices in $T$. 
The restricted isometry constant $\delta_N(I_n(\beta_0))$ is the smallest quantity such that
\[(1- \delta_{N}(I_n(\beta_0)))\|h\|_2 ^2 \leq \|A_T h\|_2 ^2 \leq (1+ \delta_N(I_n(\beta_0)))\|h\|_2 ^2,\]
for all $T \subset \{1,2,\ldots,p_n\}$ with $|T| \leq N$, where $N \leq p_n$ is an integer, and all $h \in \mathbb{R}^{|T|}$. 
The restricted orthogonality constant $\theta_{S,S'}(I_n(\beta_0))$ is the smallest quantity such that 
\[|(A_T h)^T A_{T'} h'| \leq \theta_{S,S'}(I_n(\beta_0)) \|h\|_2 \|h'\|_2\]
for all disjoint sets $T,\ T' \subset \{1,2,\ldots, p_n\}$ with $|T| \leq S,\ |T'| \leq S'$, where $S+S' \leq p_n$ 
and all vectors $h \in \mathbb{R}^{|T|}$ and $h' \in \mathbb{R}^{|T'|}$.
For $\delta_{2S}(I_n(\beta_0))$ and $\theta_{S,2S}(I_n(\beta_0))$, the UUP condition is described that 
$0 < 1-\delta_{2S}(I_n(\beta_0))-\theta_{S,2S}(I_n(\beta_0))$. 

In addition, we introduce another factor $\phi_{2S}(T_0;I_n(\beta_0))$ by 
\[\phi_{2S}(T_0;I_n(\beta_0)) := \inf \limits_{T \supset T_0,\ |T| \leq 2S,\ h \in D_{T_0,T} } 
\frac{(h^T I_n(\beta_0)h)^{\frac{1}{2}}}{\|h_T\|_2},\]
where 
\[D_{T_0,T} := \left\{h \in C_{T_0} : \|h_{T^c}\|_{\infty} \leq \min \limits_{j \in T \setminus T_0} |h_j| \right\},
\quad T \supset T_0.\]
Define that $ \min_{j \in T \setminus T_0} |h_j|= \infty$ when $T = T_0$.
The next lemma provides the asymptotic relationship between the UUP condition and a condition for 
$\phi_{2S}(T_0;I_n(\beta_0))$. 
The proof is an adaptation of that in van de Geer and B\"{u}hlmann (2007), so it is omitted.
\begin{lem}
If $\liminf_{n \rightarrow \infty} \{1 - \delta_{2S}(I_n(\beta_0)) - \theta_{S,2S}(I_n(\beta_0))\} > 0$, 
then it holds that $\liminf_{n \rightarrow \infty} \phi_{2S}(T_0;I_n(\beta_0)) > 0$.
\end{lem}
Noting that $\|h_{T_0}\|_1^2 \leq S \|h_{T_0}\|_2^2$, 
we have that 
\[\kappa(T_0;I_n(\beta_0)) \geq \phi_{2S}(T_0;I_n(\beta_0)),\]
which implies that
\[\liminf \limits_{n \rightarrow \infty} \kappa(T_0;I_n(\beta_0)) \geq \liminf_{n \rightarrow \infty}\phi_{2S}(T_0;I_n(\beta_0)) >0.\]
Noting also that $\|h_{T_0}\|_1^q \geq \|h_{T_0}\|_q^q$ for all $q \geq 1$, 
we can see that 
$\kappa(T_0;I_n(\beta_0))  \leq 2\sqrt{S} RE(T_0;I_n(\beta_0))$, and 
$\kappa(T_0;I_n(\beta_0)) \leq F_q(T_0;I_n(\beta_0))$. 
We thus have that the factors $(A),\ (B)$ and $(C)$ are strictly positive when $n$ is large 
if $\liminf_{n \rightarrow \infty} \{1 - \delta_{2S}(I_n(\beta_0)) - \theta_{S,2S}(I_n(\beta_0))\} > 0$.
So we will assume in our main theorems that these three factors are 
``asymptotically positive'', in the sense that
\[\liminf \limits_{n \rightarrow \infty} \kappa(T_0;I_n(\beta_0)) >0,\]
\[\liminf \limits_{n \rightarrow \infty} F_q(T_0;I_n(\beta_0)) >0\]
or
\[\liminf \limits_{n \rightarrow \infty} RE(T_0;I_n(\beta_0)) >0,\]
to prove the consistency of DSfPH. 

\section{Main result}
In this section,\ we will prove the consistency of DSfPH $\hat{\beta}_n$. To do this, we will prepare three lemmas. 
Lemma $4.1$ below states that the true parameter $\beta_0$ is an element of $\mathcal{B}_n$ appearing in (\ref{eq1}) with large probability 
when the sample size $n$ is large.
\begin{lem}
Put $\gamma = \gamma_{n,p_n} = K_2 \log (1+p_n) / n^\alpha$, where  
$0 < \alpha \leq 1/2$ and $K_2 >0$ are constants.
If $p_n = O(n^\xi)$ for some $\xi >1$ or if $\log p_n = O(n^\zeta)$ for some $0< \zeta < \alpha$, then it hold that 
\[\lim \limits_{n \rightarrow \infty} P(\|U_n(\beta_0)\|_\infty \geq \gamma_{n,p_n}) = 0\]
and that $\gamma_{n,p_n} \rightarrow 0$ as $n \rightarrow \infty$.
\end{lem}
\begin{proof}
Note that 
\[U_n^j(\beta_0,\tau) = \frac{1}{n} \sum \limits_{i=1}^n 
\int_0^\tau \left[\sum \limits_{k=1}^n \{Z^j_i - Z^j_k\}w_k(\beta_0,u)\right]dM_i(u),\]
where 
\[w_k(\beta_0,u) = \frac{\exp(Z_k^T \beta_0)Y_k(u)}{\sum \limits_{l=1}^n \exp(Z_l^T \beta_0)Y_l(u)}.\]
We use Lemma $2.1$ from van de Geer $(1995)$. To do this, we shall evaluate  
$\Delta U_n^j(\beta_0,u),\ u \in [0,\tau]$, and 
$\langle U_n^j(\beta_0,\cdot),U_n^j(\beta_0,\cdot) \rangle_{\tau}$.\ 
Since the jumps of $M_i$ do not occur at the same time and are all of magnitude $1$, it holds that  
\begin{align*}
|\Delta U_n^j(\beta_0,u)| 
&= \left|\frac{1}{n} \sum \limits_{i=1}^n \int_{u-}^u \left[\sum \limits_{k=1}^n \{Z^j_i - Z^j_k\}w_k(\beta_0,u)\right]dM_i(u) \right|\\
&\leq \frac{1}{n} \sup \limits_{i,j,k} |Z^j_i - Z^j_i| \sum \limits_{k=1}^n w_k(\beta_0,u)\\
&\leq \frac{2K_1}{n}.
\end{align*}
On the other hand, we have 
\begin{eqnarray*}
\lefteqn{\langle U_n^j(\beta_0,\cdot),U_n^j(\beta_0,\cdot) \rangle_{\tau}
= \frac{1}{n^2} \sum \limits_{i=1}^n \int_0^\tau \left[\sum \limits_{k=1}^n \{Z^j_i - Z^j_k\}w_k(\beta_0,u)\right]^2 d \langle M_i \rangle_u} \\
&=& \frac{1}{n^2} \sum \limits_{i=1}^n \int_0^\tau \left[\sum \limits_{k=1}^n \{Z^j_i - Z^j_k\}w_k(\beta_0,\ u)\right]^2 \exp(Z_i^T \beta_0)Y_i(u) \alpha_0(u) du \\
&\leq& \frac{1}{n^2} \sup \limits_{i,j,k}|Z^j_i - Z^j_k|^2 \sum \limits_{i=1}^n 
\int_0^\tau \exp(Z_i^T \beta_0)Y_i(u) \alpha_0(u) du \\
&\leq& \frac{4K_1^2}{n^2} n \exp(S \sup \limits_{i,j}|Z^j_i| \|\beta_0\|_{\infty}|) \int_0^\tau \alpha_0(u) du \\
&\leq&\frac{K_3}{n},
\end{eqnarray*}
where $K_3$ is a positive constant. We now use the Lemma $2.1$ from van de Geer $(1995)$:
\begin{align*}
P(|U_n^j(\beta_0,\tau)| \geq \gamma_{n,p_n})
&= P\left(|U_n^j(\beta_0,\tau)| \geq \gamma_{n,p_n},\ 
\langle U_n^j(\beta_0,\cdot),U_n^j(\beta_0,\cdot) \rangle_\tau \leq \frac{K_3}{n}\right) \\
&\leq 2 \exp \left(- \frac{\gamma_{n,p_n}^2}{2\left(\frac{2K_1}{n}\gamma_{n,p_n} + \frac{K_3}{n}\right)}\right).
\end{align*}
Write $\|\cdot\|_{\psi}$ for the Orlicz norm with respect to $\psi(x) = e^x -1$.
We apply Lemma $2.2.10$ from van der Vaart and Wellner (1996) to deduce that there exists a constant 
$L >0$ depending only on $\psi$ such that
\[\left\|\max \limits_{1 \leq j \leq p_n} |U_n^j(\beta_0,\tau)|\right\|_\psi
\leq L \left(\frac{2K_1}{n}\log (1+p_n) + \sqrt{\frac{K_3}{n} \log (1+p_n)}\right).\]
Using Markov's inequality, we have that
\begin{eqnarray*}
\lefteqn{P(\|U(\beta_0)\|_\infty \geq \gamma_{n,p_n})
= P(\max \limits_{1 \leq j \leq p_n} |U_n^j(\beta_0,\tau)| \geq \gamma_{n,p_n})} \\
&\leq& P \left(\psi \left( \frac{\max \limits_{1 \leq j \leq p_n} |U_n^j (\beta_0,\tau)|}{\|\max \limits_{1 \leq j \leq p_n} |U_n^j (\beta_0, \tau)|\|_{\psi}}\right) \geq \psi \left(\frac{\gamma_{n,p_n}} {\|\max \limits_{1 \leq j \leq p_n} |U_n^j (\beta_0,\tau)|\|_{\psi}} \right) \right) \\
&\leq& \psi \left(\frac{\gamma_{n,p_n}}{\|\max \limits_{1 \leq j \leq p_n} 
|U_n^j(\beta_0,\tau)|\|_{\psi}}\right)^{-1} \\
&\leq& \psi \left(\frac{\gamma_{n,p_n}}{L \left(\frac{2K_1}{n}\log (1+p_n) + \sqrt{\frac{K_3}{n} \log (1+p_n)}\right)}\right)^{-1}
\end{eqnarray*}
In our settings, the right-hand side of this inequality converges to $0$.
\qed
\end{proof}
Next we will show that $J_n(\beta_0)$ is approximated by $I_n(\beta_0)$.
\begin{lem}
The random sequence $\epsilon_n$ defined by
\[\epsilon_n := \|J_n(\beta_0)-I_n(\beta_0)\|_{\infty}\]
converges in probability to $0$.
\end{lem}
\begin{proof}
Define the $p_n \times p_n$ matrices $h_n(\beta_0,t)$ and $H_n(\beta_0,t)$ for $t \in [0,\tau]$ by
\begin{align*}
h_n(\beta_0,t) &:= \frac{s_n^2}{s_n^0}(\beta_0,t) - \left(\frac{s_n^1}{s_n^0}\right)^{\otimes2}(\beta_0,t), \\
H_n(\beta_0,t) &:= \frac{S_n^2}{S_n^0} (\beta_0,t) - \left(\frac{S_n^1}{S_n^0}\right)^{\otimes2} (\beta_0,t).
\end{align*}
Note that the matrices $I_n(\beta_0)$ and $J_n(\beta_0)$ can be written in this form:
\begin{align*}
I_n(\beta_0) &= \int_0^\tau h_n(\beta_0,u)s^0(\beta_0,u)\alpha_0(u)du, \\
J_n(\beta_0)&= \int_0^\tau H_n(\beta_0,u) \frac{d\bar{N}(u)}{n}.
\end{align*}
Put $\bar{M}(u) = \sum_{i=1}^n M_i(u)$. Then, it holds that 
$\|J_n(\beta_0) - I_n(\beta_0)\|_\infty \leq (I) + (I\hspace{-,1em}I) + (I\hspace{-,1em}I\hspace{-,1em}I)$, 
where
\begin{align*}
(I)&= \int_0^\tau \|H_n(\beta_0,u)-h_n(\beta_0,u)\|_\infty \frac{d\bar{N}(u)}{n}, \\
(I\hspace{-,1em}I)&=\int_0^\tau \left\|h_n(\beta_0,u)\left\{\frac{S_n^0(\beta_0,u)}{n} - s^0(\beta_0,u)\right\}\right\|_\infty \alpha_0(u) du, \\
(I\hspace{-,1em}I\hspace{-,1em}I)&=\left\|\frac{1}{n} \int_0^\tau h_n(\beta_0,u)d\bar{M}(u)\right\|_\infty.
\end{align*}
Since the process $t \leadsto \bar{N}(t)/n$ has bounded variation uniformly in $n$, 
Assumption $2.1$ implies that $(I) = o_p(1)$ and $(I\hspace{-.1em}I) =o_p(1)$.  Moreover, it follows from Assumption $2.1$ that $h_n(\beta_0,u)$ is uniformly bounded.  
So we obtain that $(I\hspace{-.1em}I\hspace{-.1em}I) = o_p(1)$ by the same way as the proof of Lemma $4.1$.
\qed
\end{proof}
The next lemma is used to control $U_n(\hat{\beta})-U_n(\beta_0)$ and $J_n(\beta_0)$. 
See Huang {\it et al}.\ (2013) and 
Hjort and Pollard (1993) for the proofs.
\begin{lem}
Define that $\eta_h = \max_{i,j} |h^T Z_i - h^T Z_j|$, for $h \in \mathbb{R}^{p_n}$. Then for all $\beta \in \mathbb{R}^{p_n}$, it holds that
\[e^{-\eta_h} h^T J(\beta) h \leq h^T[U(\beta +h) - U(\beta)] \leq e^{\eta_h}h^T J(\beta)h.\]
\end{lem}

Now, we are ready to prove the main result of this paper. Theorem $4.4$ below provides the $l_2$ consistency of DSfPH.
\begin{thm}
Under the assumption of Lemma $4.1$ and Assumption $2.1$, if $\liminf_{n \rightarrow \infty} RE(T_0;I_n(\beta_0))>0$, then
it holds that 
\[\lim \limits_{n \rightarrow \infty}P\left(\|\hat{\beta}_n - \beta_0\|^2_2 \geq \frac{K_4 \gamma_{n,p_n}}{RE^2(T_0;I_n(\beta_0))-\epsilon_n} \right) =0,\]
where $K_4$ is a positive constant and $\epsilon_n = \|I_n(\beta_0) - J_n(\beta_0)\|_{\infty} = o_p(1)$. 
In particular, $\|\hat{\beta}_n - \beta_0\|_2 \rightarrow^p 0$.
\end{thm}
\begin{proof}
It is sufficient to prove that $\|U_n(\beta_0)\|_\infty \leq \gamma_{n,p_n}$ implies 
\[\|\hat{\beta}_n - \beta_0\|^2_2 \leq \frac{K_4 \gamma_{n,p_n}}{RE^2(T_0;I_n(\beta_0))-\epsilon_n}. \]
By the construction of the estimator, we have $\|U(\hat{\beta}_n)\|_\infty \leq \gamma_{n,p_n}$, which implies 
that 
\[\|U_n(\hat{\beta}_n) - U_n(\beta_0)\|_\infty \leq \|U_n(\hat{\beta}_n)\|_\infty + \|U_n(\beta_0)\|_\infty \leq 2\gamma_{n,p_n} . \]
Note that $h := \hat{\beta}-\beta_0 \in C_{T_0}$, since it holds that
\begin{align*}
0 \geq \|\beta_0 + h\|_1 - \|\beta_0\|_1 &= \sum \limits_{j \in T_0^c} |h_{T^c_{0j}}| + \sum \limits_{j \in T_0}
 (|\beta_{0j}+h_{T_{0j}}| - |\beta_{0j}|)\\
 &\geq \sum \limits_{j \in T_0^c} |h_{T_0^c j}| -  \sum \limits_{j \in T_0} |h_{T_{0j}}| \\
 &= \|h_{T_0^c}\|_1 - \|h_{T_0}\|_1.
\end{align*}
Notice moreover that $\|h\|_1 \leq \|\hat{\beta}_n\|_1 + \|\beta_0\|_1 \leq 2 \|\beta_0\|_1$ by
the definition of $\hat{\beta}_n$.
Now, we use Lemma $4.3$ for $h$ to deduce that
\begin{align*}
h^T J_n(\beta_0)h 
&\leq e^{\eta_h} h^T [U_n(\hat{\beta}_n) - U_n(\beta_0)] \\
&\leq \exp(\max \limits_{i,j} |h^T Z_i - h^T Z_j|) \cdot 2\gamma_{n,p_n} \|h\|_1 \\
&\leq \exp(4K_1 \|\beta_0\|_1) \cdot 4\gamma_{n,p_n} \|\beta_0\|_1 \\
&=: K_4 \gamma_{n,p_n}.
\end{align*}
Thus it holds that
\begin{align*}
h^T I_n(\beta_0) h 
&\leq |h^T (I_n(\beta_0) - J_n(\beta_0))h| + h^T J_n(\beta_0) h \\
&\leq \epsilon_n h^T h + K_4 \gamma_{n,p_n} \\
&= \epsilon_n \|\hat{\beta}_n - \beta_0\|^2_2 + K_4 \gamma_{n,p_n}.
\end{align*}
By the definition of the restricted eigenvalue, we have that 
\begin{align*}
RE^2(T_0;I_n(\beta_0))
&\leq \frac{h^T I_n(\beta_0) h}{\|\hat{\beta}_n-\beta_0\|_2^2} \\
&\leq \frac{\epsilon_n \|\hat{\beta}_n - \beta_0\|_2^2 + K_4 \gamma_{n,p_n}}{\|\hat{\beta}_n-\beta_0\|_2^2}.
\end{align*}
Noting that $RE^2(T_0;I_n(\beta_0)) >0$, we obtain that 
\[\|\hat{\beta}_n - \beta_0\|_2^2 \leq \frac{K_4 \gamma_{n,p_n}}{RE^2(T_0;I_n(\beta_0))- \epsilon_n}.\]
\qed
\end{proof}
To derive the $l_1$ consistency and the $l_q$ consistency of DSfPH, we shall use 
the compatibility factor and the weak cone invertibility factor, respectively. 
\begin{thm}
Under the assumptions of Lemma $4.1$ and Assumption $2.1$, if $\liminf_{n \rightarrow \infty} \kappa(T_0;I_n(\beta_0)) >0$,  then the following $(i)$ and $(ii)$ hold true.
\begin{description}
\item[$(i)$]
It holds that
\[
\lim \limits_{n \rightarrow \infty}P\left(\|\hat{\beta}_n - \beta_0 \|_1 \geq \frac{4K_5 S \gamma_{n,p_n}}{\kappa^2(T_0;I_n(\beta_0)) - 4S \epsilon_n}\right) 
=0,\]
where $K_5$ is a positive constant. In particular, 
$\|\hat{\beta}_n - \beta_0\|_1 \rightarrow^p 0$. 
\item[$(ii)$]
It holds for any $q >1$ that 
\[
\lim \limits_{n \rightarrow \infty} 
P \left(\|\hat{\beta}_n - \beta_0\|_q \geq
\frac{2S^{\frac{1}{q}} \epsilon_n}{F_q(T_0;I_n(\beta_0))} \cdot \frac{2K_5 S \gamma_{n,p_n}}{\kappa^2(T_0;I_n(\beta_0)) -2S \epsilon_n} + \frac{2K_5 S^{\frac{1}{q}} \gamma_{n,p_n}}{F_q(T_0;I_n(\beta_0))}
\right)
=0.
\]
In particular, $\|\hat{\beta}_n - \beta_0\|_q \rightarrow^p 0$.
\end{description}
\end{thm}
\begin{proof}
It follows from the proof of Theorem $4.4$ that 
\[h^T J_n(\beta_0) h \leq K_5 \gamma_{n,p_n} \|\hat{\beta}_n-\beta_0\|_1.\]
Noting that $\|b\|^2_2 \leq \|b\|^2_1$ for all $b \in \mathbb{R}^{p_n}$, we have that
\begin{align*}
h^T I_n(\beta_0) h
&\leq \epsilon_n \|\hat{\beta}_n - \beta_0\|^2_2 + K_5 \gamma_{n,p_n} \|\hat{\beta}_n - \beta_0\|_1 \\
&\leq \epsilon_n \|\hat{\beta}_n-\beta_0\|^2_1 + K_5 \gamma_{n,p_n} \|\hat{\beta}_n - \beta_0\|_1.
\end{align*}
The definition of $\kappa(T_0;I_n(\beta_0))$ implies that 
\begin{align*}
\kappa^2(T_0;I_n(\beta_0))
&\leq \frac{S h^T I_n(\beta_0)h}{\|h_{T_0}\|^2_1} \\
&\leq \frac{S \epsilon_n \|h\|^2_1 + K_5 S \gamma_{n,p_n} \|h\|_1}{\|h_{T_0}\|^2_1}.
\end{align*}
Since $\|h\|_1 \leq 2 \|h_{T_0}\|_1$, this yields the conclusion in (i).

On the other hand, using the weak cone invertibility factor for every $q \geq 1$, we have that 
\[F_q (T_0;I_n(\beta_0)) 
\leq  \frac{S^{\frac{1}{q}}\epsilon_n \|h\|^2_1 + S^{\frac{1}{q}}K_5 \gamma_{n,p_n} \|h\|_1}
{\|h_{T_0}\|_1 \|h\|_q},
\]
which implies that
\[
\|\hat{\beta}_n - \beta_0\|_q \leq
\frac{2S^{\frac{1}{q}}\epsilon_n \|\hat{\beta}-\beta_0\|_1 + 2S^{\frac{1}{q}}K_5 \gamma_{n,p_n}}{F_q(T_0;I_n(\beta_0))}.
\]
Using the $l_1$ bound derived above, we obtain the conclusion in (ii).
\qed
\end{proof}
\vskip 20pt
{\bf Acknowledgements.}
The second author's work was supported by Grant-in-Aid for Scientific Research (C), 15K00062, from Japan Society for the Promotion of Science.

\end{document}